\documentclass[12pt]{amsart}
\usepackage{amsfonts}
\usepackage{amsmath}
\usepackage{amssymb}
\usepackage{caption}
\usepackage[margin=1in]{geometry}
\usepackage{graphicx}
\usepackage{hyperref}

\usepackage{xcolor}

\newcommand{\R}{\mathbb{R}}
\newcommand{\C}{\mathbb{C}}
\newcommand{\N}{\mathbb{N}}
\renewcommand{\H}{\mathbb{H}}

\renewcommand{\L}{\mathcal{L}}
\renewcommand{\S}{\mathcal{S}}

\newtheorem{theorem}{Theorem}[section]
\newtheorem{corollary}[theorem]{Corollary}
\newtheorem{lemma}[theorem]{Lemma}
\newtheorem{proposition}[theorem]{Proposition}

\DeclareMathOperator{\Real}{Re}
\DeclareMathOperator{\Imaginary}{Im}
\DeclareMathOperator\supp{supp}
\DeclareMathOperator\spn{span}

\theoremstyle{definition}

\newtheorem{remark}[theorem]{Remark}

\numberwithin{equation}{section}

\begin{document}
\title[On a space of functions with entire Laplace transforms]{On a space of functions with entire Laplace transforms and its connection with the optimality of the Ingham-Karamata theorem}

\author[M. Callewaert]{Morgan Callewaert}
\address{Department of Mathematics: Analysis, Logic and Discrete Mathematics\\ Ghent University\\ Krijgslaan 281\\ 9000 Gent\\ Belgium}
\email{morgan.callewaert@UGent.be}

\author[L. Neyt]{Lenny Neyt}
\thanks{L. Neyt gratefully acknowledges support from the Alexander von Humboldt Foundation and by the Research Foundation--Flanders through the postdoctoral grant 12YG921N. His research was funded in part by the Austrian Science Fund (FWF) 10.55776/ESP8128624. For open access purposes, the author has applied a CC BY public copyright license to any author-accepted manuscript version arising from this submission.}
\address{University of Vienna\\ Faculty of Mathematics\\ Oskar-Morgenstern-Platz 1 \\ 1090 Wien\\ Austria}
\email{lenny.neyt@univie.ac.at}

\author[J. Vindas]{Jasson Vindas}
\thanks{The work of J. Vindas was supported by the Research Foundation--Flanders through the FWO-grant G067621N and by Ghent University through the grant number bof/baf/4y/2024/01/155}
\address{Department of Mathematics: Analysis, Logic and Discrete Mathematics\\ Ghent University\\ Krijgslaan 281\\ 9000 Gent\\ Belgium}
\email{jasson.vindas@UGent.be}

\begin{abstract}  We study approximation properties of the Fr\'{e}chet space of all continuously differentiable functions $\tau$ such that $\tau'(x)=o(1)$
 and such that their Laplace transforms admit entire extensions to $\mathbb{C}$. As an application, these approximation results are combined with the open mapping theorem to show the optimality theorem for the Ingham-Karamata Tauberian theorem.
 
\end{abstract}
\maketitle

\begin{center}

\emph{Dedicated to the memory of Franciscus Sommen}

\end{center}

\bigskip

\section{Introduction}
Let $V_{-\infty}$ be the function space consisting of all those $\tau\in C^{1}[0,\infty)$ such that 
\begin{equation}
\label{eq:ineq1}
\tau'(x)=o(1) , \qquad x \to \infty ,
\end{equation}
holds and such that its Laplace transform
\[ \L\{\tau; s\} = \int_{0}^{\infty} \tau(x) e^{-sx} dx \]
admits extension to $\mathbb{C}$ as an entire function. 

In this article we are interested in approximation properties of the space $V_{-\infty}$ with respect to the family of norms 
\begin{equation}
\label{eq:ineq2}  \| \tau \|_{\alpha, r}= \| \tau' \|_{L^{\infty}} +\sup_{s \in \overline{R}_{\alpha,r}} | \L\{\tau;s\}|,
\end{equation}
where the parameters satisfy $\alpha<0$ and $r>0$ and $R_{\alpha, r}$ stands for the open rectangle $\{ s \in \C \mid  \alpha<\Real{s} <1 \text{ and } |\Imaginary{s}| < r \}$. The approximation features we shall study are intimately connected with the optimality of the Ingham-Karamata theorem, as we now proceed to explain.

The Ingham-Karamata theorem\footnote{This Tauberian theorem is sometimes also known as Newman's theorem or the Fatou-Riesz theorem for Laplace transforms.} is a foundational result in complex Tauberian theory for Laplace transforms. In its simplest form, it states:

\begin{theorem}\label{Th: I-K} Let $\tau$ be Lipschitz continuous on $[0,\infty)$. If its Laplace transform 
$\mathcal{L}\{\tau ; s\}$ has an analytic continuation across the imaginary axis, then
\begin{equation}
\label{eq:1I-K}
\tau(x)=o(1) .
\end{equation}
\end{theorem}

This classical result and its many generalizations have important applications in diverse areas of mathematics such as number theory and operator theory. In particular, numerous developments on the subject from the last four decades have been motivated by the theory of operator semigroups and its applications in stability theory for various classes of evolution equations. We refer to the monographs \cite{A-B-H-N,C-Qbook, korevaarbook} for accounts on the Ingham-Karamata theorem and related complex Tauberian theorems; see also the articles \cite{Chill-Seifert2016,d-vOptIngham,D-V2019,Stahn2018} for some recent results.

It is natural to attempt to improve the decay rate of $\tau$ in \eqref{eq:1I-K} by strengthening the assumptions on the Laplace transform from Theorem \ref{Th: I-K}. Such a problem has been extensively studied in the literature. It is known (cf. \cite{A-B-H-N,Chill-Seifert2016, Stahn2018}) that quantified information about the shape of the region of analytic continuation and growth assumptions on such a region for the Laplace transform lead together to quantitative remainders in the Ingham-Karamata theorem. However, when bounds on the Laplace transform are absent, it is impossible to get a better remainder than $o(1)$ in \eqref{eq:1I-K} even if one assumes analytic continuation of $\mathcal{L}\{\tau ; s\}$ to a whole half-plane containing the imaginary axis. This was first studied in \cite{D-V-NoteAbsenceRemWienerIkehara}, refuting a conjecture from \cite{muger}. The stronger hypothesis of entire extension does not yield any reasonable remainder either, as stated by the following theorem.

\begin{theorem}\label{th3 abserrorW-I} 
Let $\rho$ be an arbitrary positive function tending to 0. There is always $\tau\in V_{-\infty}$ such that
\begin{equation}
\label{eq:2I-K}
\limsup_{x\to\infty} \frac{|\tau(x)|}{\rho(x)}=\infty.
\end{equation}
\end{theorem}

A slightly weaker version  (but still strong enough to yield the optimality of the classical Ingham-Karamata theorem) of Theorem \ref{th3 abserrorW-I} was shown in \cite{B-D-V-AbsenceRemWienerIkeharaConstructive}, with the condition \eqref{eq:ineq1} replaced by $\tau'(x) = O(1)$ and where the function $\tau$ was explicitly constructed. 

Theorem \ref{th3 abserrorW-I} might essentially be proved via the open mapping approach developed in \cite{D-V-NoteAbsenceRemWienerIkehara}. We have however discovered a gap in the method from \cite{D-V-NoteAbsenceRemWienerIkehara}. In fact, one key step in such a method consists in establishing that a certain inequality (arising from application of the open mapping theorem) may be applied to continuous differentiable functions on $[0,\infty)$ that satisfy \eqref{eq:ineq1} and whose Laplace transforms have analytic continuation to a neighborhood of one of the rectangles $\overline{R}_{\alpha,r}$, which in the context of \cite{D-V-NoteAbsenceRemWienerIkehara} would follow if all such functions belong to the completion of $V_{-\infty}$ with respect to the norm \eqref{eq:ineq2}. It turns out that the approximation net proposed in \cite{D-V-NoteAbsenceRemWienerIkehara} fails to deliver the latter claimed approximation property. The aim of this paper is to supply a proof of this crucial approximation claim, amending so the open mapping method from \cite{D-V-NoteAbsenceRemWienerIkehara}. 

We would like to point out that the open mapping approach has further played a central role in establishing optimality results for quantified versions of the Ingham-Karamata theorem, see \cite{DS2019IJM, DS2019AM}. We also mention that the use of functional analysis arguments to produce counterexamples in Tauberian theory goes back to Ganelius and, according to himself \cite[p.~3]{ganelius}, was first suggested by H\"{o}rmander.

Finally, it is also worth noticing that Theorem \ref{th3 abserrorW-I} yields the optimality result for the classical Wiener-Ikehara theorem \cite{korevaarbook}, originally obtained in \cite{D-V-NoteAbsenceRemWienerIkehara} (see also \cite{B-D-V-AbsenceRemWienerIkeharaConstructive} for a constructive proof).
\begin{corollary}\label{OW-I}
Let $\rho$ be an arbitrary positive function tending to 0. There is a non-decreasing function $S$ on $[0,\infty)$
 such that its Laplace-Stieltjes transform 
 \[\mathcal{L}\{dS;s\}=\int_{{0}^{-}}^{\infty} e^{-s x} dS(x) \qquad \mbox{converges for }\Real s>1\]
and 
 \[ 
\mathcal{L}\{dS;s\} - \frac{a}{s-1}
 \]
 admits extension to $\mathbb{C}$ as an entire function for some $a>0$, but such that
 \[
 \limsup_{x\to\infty} \frac{|S(x)-ae^{x}|}{\rho(x)e^{x}}=\infty.
 \] 
  \end{corollary}
\begin{proof}
Let $\tau$ be a function satisfying the properties from Theorem \ref{th3 abserrorW-I}. Find $a > 0$ large enough such that $\tau(x)+\tau'(x)+a\geq 0$ for all $x\geq0$. The non-decreasing function $S(x)=\tau(x)e^{x}+ae^{x}$ then satisfies all requirements in view of \eqref{eq:2I-K} and the fact that its Laplace-Stieltjes transform is given by $\mathcal{L}\{dS;s\}=s \mathcal{L}\{\tau;s-1\}+a/(s-1)$.
\end{proof}

The plan of the article is as follows. Our main result, Theorem \ref{MainResult}, is stated in the next section, where we use some functional analysis arguments to reduce it to the intermediate approximation property stated in Proposition \ref{p:DensitySteps}. The proof of Proposition \ref{p:DensitySteps} is given in Section \ref{DenseRange}, while that of 
Theorem \ref{th3 abserrorW-I} is discussed in Section \ref{sect absence of remainders}.

\section{Statement of the main result}
\label{MainResults}

This section is dedicated to stating our main result. In preparation, we first need to introduce some notation and function spaces. 

We start by extending the definition of the space $V_{-\infty}$ given in the Introduction. Let $\beta\in [-\infty,0)=(-\infty,0)\cup\{-\infty\}$. We write $\H_{\beta} = \{ s \in \C \mid \Real{s} > \beta \}$ and then define
	\[ V_{\beta} = \{\tau \in C^{1}[0,\infty) \mid \tau \text{ satisfies \eqref{eq:ineq1} and } \L\{\tau;s\} \text{ has analytic extension to } \mathbb{H}_\beta\} . \]
We topologize $V_{\beta}$ with the fundamental family of norms $\{ \|\cdot\|_{\alpha, r} \mid \alpha > \beta , r > 0 \}$, see \eqref{eq:ineq2}, thereby inducing a Fr\'{e}chet space topology on it. Given fixed $\alpha<0$ and $r > 0$,  we also consider the space $V_{\alpha, r}$ consisting of those functions $\tau \in C^{1}[0,\infty)$ for which \eqref{eq:ineq1} holds and such that 
$\L\{\tau; s\}$ has analytic continuation to the open rectangle $R_{\alpha, r}$ and continuous extension to the closed rectangle $\overline{R}_{\alpha,r}$; 
it is a Banach space endowed with the norm \eqref{eq:ineq2}.

For $\alpha, \alpha_1, \alpha_2, \beta, \beta_1, \beta_2 <0$ and $r, r_1, r_2 > 0$, we obviously have the following natural continuous inclusions:
\begin{align*}
	V_{\beta_1} \subseteq V_{\beta_2} &\quad \Longleftrightarrow \quad \beta_1 \leq \beta_2 , \\
	V_{\alpha_1, r_1} \subseteq V_{\alpha_2, r_2} &\quad \Longleftrightarrow \quad \alpha_1 \leq \alpha_2 \text{ and } r_{1} \geq r_{2} , \\
	V_\beta \subseteq V_{\alpha, r} &\quad \Longleftrightarrow \quad \beta < \alpha.
\end{align*} 
The purpose of this paper is to show that, in fact, all of these inclusions are dense.
In particular, the density of $V_{-\infty}$ then plays a special role, being contained in every other space.
Our main result reads as follows.

\begin{theorem}\label{MainResult}
    The space $V_{-\infty}$ is dense in the spaces $V_{\alpha, r}$ and $V_{\alpha}$ for any $\alpha <0$ and $r > 0$.
\end{theorem}
 
The largest share of the work will be in showing the following intermediate step, whose proof will be given in Section \ref{DenseRange}. 

\begin{proposition}\label{p:DensitySteps}
	The space $V_{\beta}$ is dense in $V_{\alpha, r}$ for every $-\infty < \beta < \alpha<0$ and $r > 0$.
\end{proposition}

Let us now end this section by explaining how this particular case can directly be used to prove Theorem \ref{MainResult}. 
To this end, we recall some classical facts about projective spectra of topological vector spaces.
Let $(X_n)_{n \in \N}$ be a sequence of topological vector spaces and let $u^{n+1}_{n} : X_{n+1} \to X_n$ be a continuous linear mapping for each $n \in \N$. Consider the projective limit of the spectrum $(X_n)_{n \in \N}$, that is,
\[X=\left\{(x_n)_{n\in \mathbb{N}}\in \prod_{n \in \N} X_n\, \Big\vert \, x_n=u^{n+1}_{n}(x_{n+1}), n\in\N\right\},\]
with the natural projection mappings  $u_j : X \to X_j : (x_n)_{n \in \N} \to x_j$ and the projective topology with respect to them. The spectrum is called \emph{reduced} if $u_j$ has dense range for each $j \in \N$. 
We will use the following result due to De Wilde.

	\begin{lemma}[{\cite{DW-CritDensSepLimProjInd}}]
		\label{l:DeWilde}
		 Let $(X_n)_{n\in\N}$ be a spectrum of complete metrizable topological vector spaces. If $u_n^{n+1}$ has dense range for each $n\in\N$, then the spectrum $(X_{n})_{n\in\N}$ is reduced.
	\end{lemma}

	\begin{proof}[Proof of Theorem \ref{MainResult}]
		First, we observe that $V_{\beta}$ is dense in $V_{\alpha}$ whenever $\beta < \alpha$. Indeed, this directly follows  from Proposition \ref{p:DensitySteps} as $V_{\alpha}$ is isomorphic to the projective limit of $(V_{\alpha + \frac{1}{n + 1}, n+1})_{n \in \N}$ with natural inclusions.
		Then, Lemma \ref{l:DeWilde} yields that $V_{-\infty}$ is dense in $V_{\beta}$ for any $\beta \in \R$ as $V_{-\infty}$ is isomorphic to the projective limit of the spectrum $(V_{\beta - n})_{n \in \N}$ with natural inclusions. 
		Hence, we find the chain of dense inclusions $V_{-\infty} \subseteq V_{\beta} \subseteq V_{\alpha, r}$ whenever $\beta < \alpha<0$ and $r > 0$.
	\end{proof}

\section{The proof of Proposition \ref{p:DensitySteps}}\label{DenseRange}

In this section, we consider the proof of Proposition \ref{p:DensitySteps} and, as a consequence, we will complete that of Theorem \ref{MainResult}. We fix throughout this section $\beta<\alpha<0$ and $r > 0$ and set $R = \overline{R}_{\alpha,r}$.

We start by making a reduction. Consider the subspace
	\[ V^{0}_{\alpha, r} = \{ \tau \in V_{\alpha, r} \mid \tau(0) = \tau'(0) = 0 \} . \]
To show Proposition \ref{p:DensitySteps}, it suffices to prove that $V^{0}_{\alpha, r}$ is contained in the closure of $V_{\beta}$ in $V_{\alpha, r}$.
Indeed, note that any $\tau \in V_{\alpha, r}$ can be written as $\tau = \tau_{0} + \tau_{1}$ with $\tau_{0} \in V^{0}_{\alpha, r}$ and $\tau_{1}$ having compact support, so that in particular $\tau_{1} \in V_{-\infty}$, whence our claim follows.

Let us now introduce the following auxiliary space. Here, $\S(\R_{+})$ stands for the subspace of the Schwartz spaces $\S(\R)$ \cite{S-ThDist} whose elements have support contained in the half-line $[0, \infty)$.  
 \begin{multline*} 
 	Y = \{ \tau \in \S(\R_{+}) \mid \exists \alpha' < \alpha, ~ r' > r : \\ \L\{\tau;s\} \text{ has analytic extension to a neighborhood of } \overline{R}_{\alpha', r'} \} . 
\end{multline*}
We shall first show  that $Y$ is dense in $V^{0}_{\alpha, r}$. For it, one of our approximation steps makes use of the next result, which is a particular case of 
\cite[Theorem 5.3]{D-V2019}.

\begin{lemma}[{\cite{D-V2019}}] \label{l:Tauberian} Let $\tau\in V_{\alpha,r}$. Then\footnote{The conclusion $\tau\in L^{\infty}(\mathbb{R})$ actually suffices for our purposes and the latter follows from both \cite[Theorem 3.1]{D-V2019} and the finite form of the Ingham-Karamata theorem \cite{d-vOptIngham}}, $\lim_{x\to\infty}\tau(x)=0$.
\end{lemma}

\begin{lemma}\label{LemmaMisDense}
    $Y$ is dense in $V^{0}_{\alpha, r}$.
\end{lemma}
\begin{proof}
Our proof consists of three steps, each time strengthening the assumptions we may impose on the functions. Let us therefore fix some $\tau \in V^{0}_{\alpha, r}$ and $\varepsilon > 0$.

\underline{Step 1} : There exists a function $g \in C^{1}[0,\infty)$ for which  \eqref{eq:ineq1} holds, $g(0) = g'(0) = 0$, $\L\{g; s\}$ has analytic continuation to an open neighborhood $\Omega$ of $R$, and $\|\tau - g\|_{\alpha, r} \leq \varepsilon / 3$.
Define, for $\lambda > 0$, the function $\tau_{\lambda}(x) = \tau((1 + \lambda) x)$.
Then $\L\left\{\tau_{\lambda}; s\right\} = (1 + \lambda)^{-1} \L\{\tau; s/(1 + \lambda)\}$. Consequently, $\L\{\tau_{\lambda}; s\}$ has analytic extension to $(1 + \lambda) \mathring{R}$, an open neighborhood of \(R\).
Since $\tau'(x) = o(1)$, we obtain that
\[ \lim_{\lambda\to0^{+}} \| \tau' - \tau'_{\lambda} \|_{L^{\infty}} =0 . \]
On the other hand, we also clearly have $\lim_{\lambda\to0^{+}} \L\{\tau_\lambda; s\}= \L\{\tau; s\}$ uniformly for $s\in R$. Setting $g = \tau_{\lambda}$ gives the desired function if $\lambda$ is sufficiently small.

\underline{Step 2} : There exists a function $f \in C^{1}[0,\infty)$ for which \eqref{eq:ineq1} holds, $f(0) = f'(0) = 0$, $\L\{f; s\}$ has smooth extension to $i\R$ with the property that $\L\{f; it\}$ and all its derivatives are bounded, $\L\{f; s\}$ has analytic extension to a neighborhood of $R' = \overline{R}_{\alpha', r'} \supset R$, for certain $\alpha' < \alpha$ and $r' > r$, and $\|g - f\|_{\alpha, r} \leq \varepsilon / 3$.

Consider a function $\phi \in \S(\R)$ such that $\phi(0) = 1$ and its Fourier transform 
	\[ \widehat{\phi}(\xi)= \int_{-\infty}^{\infty} \phi(x)e^{-i \xi x}dx \]
 is an even function with compact support contained in $[-A, A]$ for $A > 0$.
Define for each $\lambda > 0$ the function $g_{\lambda}(x) = g(x) \phi_{\lambda}(x)$ with $\phi_\lambda(x)=\phi(\lambda x)$.
Then $\|g'-g'_\lambda\|_{L^\infty}\to0$ as $\lambda \to 0^{+}$.
Indeed, this follows immediately from the definition of $\phi_{\lambda}$, the property $g'(x) = o(1)$, and the fact that $g$ is bounded (by Lemma \ref{l:Tauberian}). 
Since $g_{\lambda}$ has rapid decay, its Fourier transform $\widehat{g}_{\lambda}(\xi)= \L\{g_{\lambda}; i\xi\}$ is smooth with bounded derivatives.
Let us now consider the analytic continuation. We have $\widehat{\phi}_{\lambda}(\xi) = \lambda^{-1} \widehat{\phi}(\xi / \lambda)$ and $\supp \widehat{\phi}_{\lambda} \subseteq [-\lambda A, \lambda A]$. 
Set $\Delta = \min_{x \in R} d(x,\partial \Omega) > 0$, then $\supp \widehat{\phi}_{\lambda} \subseteq [-\Delta / 2, \Delta / 2]$ for $\lambda$ small enough.
By Plancherel's theorem, for $\Real{s} > 0$, we obtain that
        \[ \L\{g_\lambda; s\} = \int^{\infty}_0 g(x)\phi_\lambda(x) e^{-sx} dx = \frac{1}{2 \pi \lambda} \int^{\Delta / 2}_{-\Delta / 2} \L\{g; s + i\xi\} \widehat{\phi}\left(\frac{\xi}{\lambda}\right) d\xi . \]
If we take $\alpha' = \alpha - \Delta / 4$ and $r' = r + \Delta / 4$, then the right-hand side of the above expression shows that there is a neighborhood of $R' = \overline{R}_{\alpha', r'}$ where $\L\{g_{\lambda}; s\}$ admits analytic continuation if $\lambda$ is small enough.
The family $\{ \L\{g_{\lambda}; s\} \mid \lambda \text{ small enough} \}$ is uniformly bounded on $\mathring{R}'$, since
        \begin{align*}
        \sup_{s \in \mathring{R}'} | \L\{g_{\lambda}; s\} |
        &= \frac{1}{2 \pi \lambda} \sup_{s \in R'} | \int^{\lambda A}_{-\lambda A} \L\{g; s + i\xi\}   \widehat{\phi}\left(\frac{\xi}{\lambda}\right)d\xi | \\
        &\leq \frac{A}{\pi} \| \widehat{\phi} \|_{L^{\infty}} \sup_{s \in [\alpha', 1] + i [-r-\frac{3\Delta}{4}, r+\frac{3\Delta}{4}]} |\L\{g; s\}| 
        <\infty ,
        \end{align*}
 as $[\alpha', 1]+i[-r-\frac{3\Delta}{4}, r+\frac{3\Delta}{4}]\subseteq \Omega$. 
 Now, by the mean-value theorem,
 \[
 |\L\{g_{\lambda}; s\} - \L\{g; s\}|\leq \lambda  \| g\|_{L^\infty} \|\phi'\|_{L^{\infty}}\int_{0}^{\infty}x e^{-x\Real{s}}dx\to 0, \]
for  $\Real{s} > 0$ as  $\lambda \to 0^{+}$. Then, we may conclude from Montel's theorem that
 	\[ \L\{g_{\lambda}; s\} \to \L\{g; s\}\qquad  \mbox{uniformly for }s\in R, \]
because $R$ is a compact subset of $\mathring{R}'$. Thus, we may take $f = g_{\lambda}$ for $\lambda$ small enough.

 \underline{Step 3} : There exists a function $\psi \in Y$ such that $\|f - \psi\|_{\alpha, r} \leq \varepsilon / 3$.
 
 In this final step, we take $\varphi \in C^{\infty}(\R)$ with compact support in $[0, \infty)$ such that $\int_{0}^{\infty} \varphi(x) dx = 1$. 
Consider then $\varphi_{\lambda}(x) = \lambda^{-1} \varphi(\lambda^{-1} x)$ for $\lambda > 0$ and the functions $f_{\lambda} = f * \varphi_{\lambda}$.
 Now, $\widehat{f}_{\lambda}(\xi) = \widehat{f}(\xi) \cdot \widehat{\varphi}_{\lambda}(\xi) \in \S(\R)$, so that $f_{\lambda} \in \S(\R_{+})$. Moreover, $\L\{f_{\lambda}; s\} = \L\{f; s\} \L\{\varphi_{\lambda}; s\}$ has analytic extension to a neighborhood of $R'$. Using that $f'(0) = 0$ and $f'(x)=o(1)$,
 one now easily verifies that $f_{\lambda} \to f$ as $\lambda \to 0^{+}$ in $V_{\alpha, r}$. Consequently, we may choose $\psi = f_{\lambda}$ for $\lambda$ small enough.
 
 By taking $g, f, \psi$ as in the previous steps with respect to some $\varepsilon > 0$, we obtain $\|\tau - \psi\|_{\alpha, r} \leq \varepsilon$ and our proof is complete. 
\end{proof}

In view of the previous lemma, to show Proposition \ref{p:DensitySteps} it suffices to show that $Y$ lies in the closure of $V_{\beta}$ in $V_{\alpha, r}$ for any $\beta < \alpha$. 
 We shall now introduce another approximation step that allows us to work with functions whose Laplace transforms have analytic extension to a certain concrete region containing the imaginary axis. Let $\tau \in Y$. Then, consider $\tau_{h}(x) = e^{-h x} \tau(x)$ for any $h > 0$. As $\L\{\tau_{h}; s\} = \L\{\tau; s + h\}$, it follows that $\L\{\tau_{h}; s\}$ has analytic extension to a neighborhood of $-h + \overline{R}_{\alpha', r'}$ as well as to $\H_{-h}$ with continuous extension to $\overline{\H}_{-h}$ for certain $\alpha' < \alpha$ and $r' > r$, which in particular contains the imaginary line.
Now, an easy calculation shows that $\tau_{h} \to \tau$ as $h \to 0^{+}$ in $V_{\alpha, r}$. Hence, we may just concern ourselves with the functions $\tau_{h}$ and we will always assume that $0<h < - \alpha'$. 

From now on we fix some $\tau \in Y$ and let $\beta < \alpha' < \alpha$ and $r' > r$ be such that $\L\{\tau; s\}$ has analytic extension to a neighborhood of $R' = \overline{R}_{\alpha', r'}$.
Let $\Gamma_h$, $h > 0$, be the contour obtained via the union of $-h + i(-\infty, -r']$, $\partial R' \cap \{s \in \C \mid \Real{s} \leq -h \}$, and $-h + i[r', \infty)$, oriented as in Figure \ref{Contour1Gev2}.
Note that $\Real{s} \leq -h < 0$ for any $s \in \Gamma_{h}$. Then $\tau_h$ has the following integral representation. 

\captionsetup[figure]{labelfont={bf},labelformat={default},labelsep=quad,name={Fig.}}
\begin{figure}[h!]
\centering
		\begin{minipage}{.5\textwidth}
  			\centering
  			\includegraphics[scale=0.55]{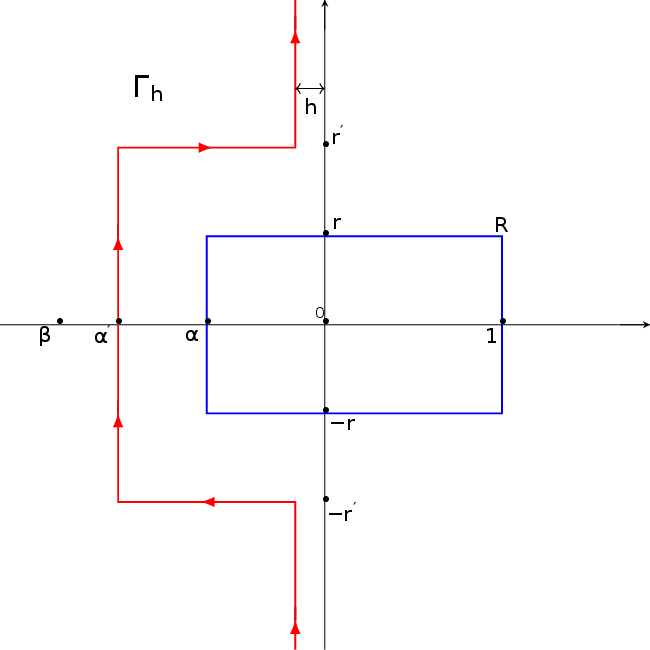}
			\caption{The contour $\Gamma_h$}
  		  			\label{Contour1Gev2}
		\end{minipage}
		\begin{minipage}{.45\textwidth}
			\centering
    			\includegraphics[scale=0.55]{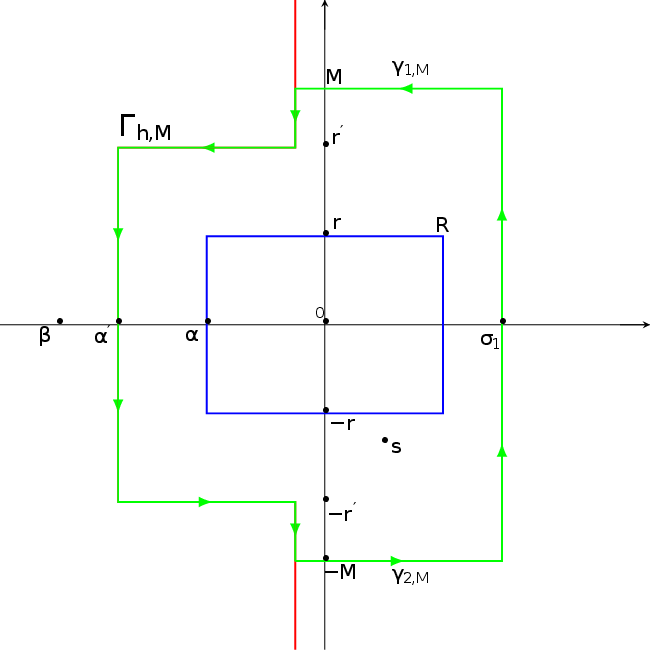}
			\caption{The contour $\Gamma_{h, M}$}
			\label{ContourGammahM}
		\end{minipage}
\end{figure}

\begin{lemma}\label{LemmaIntegralrepresentation}
For $h \geq 0$, we have
			\begin{equation}
				\label{IntegralTau}
				\tau_h(x) = \frac{1}{2 \pi i }\int_{\Gamma_h} \L\{\tau_h; \zeta\} e_{\zeta}(x)\, d\zeta,
			\end{equation}
where $e_{\zeta}(x) $ denotes the restriction of $e^{\zeta x}$ to $[0,\infty)$.
\end{lemma}

\begin{proof}  
It suffices to show that, for all $s \in R \cup \H_{0}$,
    \begin{equation}
    	\label{eq:IntegralRep}
        \L\{\tau_h; s\} = \frac{1}{2\pi i } \int_{\Gamma_h} \frac{\L\{\tau_h; \zeta\}}{(s-\zeta)}d\zeta.
    \end{equation}
In fact, we have
			\[ \L\{e_{\zeta}; s\} = \frac{1}{(s-\zeta)} \]
for $\zeta \in \C$ with $\Real{\zeta} \leq 0$. Then \eqref{IntegralTau} would directly follow by applying the inverse Laplace transform to \eqref{eq:IntegralRep}.
Fix $s \in R \cup \H_{0}$. Take any $\sigma_{1} > 1$ and $M > r'$.
Put $\Gamma_{h, M} = \Gamma_{h} \cap \{ s \in \C \mid |\Imaginary{s}| \leq M \}$, $\gamma_{1, M} = [-h, \sigma_{1}] + iM$, and $\gamma_{2, M} = [-h, \sigma_{1}] - iM$, where we use the orientation as in Figure \ref{ContourGammahM}.
Since $\L\{\tau_h; s\}$ is analytic on the domain at the right-hand side of the contour $\Gamma_h$ and continuous on its boundary, it follows from Cauchy's integral formula that
		\[  \L\{\tau_h; s\} = \frac{1}{2 \pi i} \Big(\int_{\Gamma_{h, M}} + \int_{\gamma_{1, M}} + \int_{\gamma_{2, M}} + \int_{\sigma_{1} + i[-M, M]}\Big) \frac{\L\{\tau_h; \zeta\}}{(\zeta-s)}  d\zeta . \]
Clearly,
		\[ \frac{1}{2 \pi i} \int_{\Gamma_{h, M}} \frac{\L\{\tau_h; \zeta\}}{(\zeta-s)} d\zeta \to \frac{1}{2 \pi i} \int_{\Gamma_{h}}\frac{\L\{\tau_h; \zeta\}}{(\zeta-s)} d\zeta , \qquad \text{as } M \to \infty . \]
Hence, to show \eqref{eq:IntegralRep}, we need to verify that the other integrals vanish as $M \to \infty$.
Since $\L\{\tau_{h}; s\}$ has uniform fast decay on the strip $\{ s \in \C \mid -h \leq \Real{s} \leq \sigma_{1} \}$, one sees that
		\[ \int_{\gamma_{j, M}} \frac{\L\{\tau_h; \zeta\}}{(\zeta-s)} d\zeta \to 0, \quad \text{ as } M \to \infty , \qquad j = 1, 2 .  \]
On the other hand, we obtain that
		\[ \lim_{M \to \infty} \int_{\sigma_{1} + i[-M, M]} \frac{\L\{\tau_h; \zeta\}}{(\zeta-s)} d\zeta = \int_{-\infty}^{\infty} \frac{\L\{\tau_h; \sigma_{1} + i\xi\}}{(\sigma_{1} + i\xi - s)}d\xi \]
is well-defined. As $\sigma_{1}$ was taken arbitrarily, the above expression is independent of $\sigma_1$. Furthermore, since $\sup_{\Real{s} \geq 0} |s \L\{\tau_{h}; s\}| < \infty$, an application of the Lebesgue dominated convergence theorem proves that
		\[ \int_{-\infty}^{\infty} \frac{\L\{\tau_h; \sigma_{1} + i\xi\}}{(\sigma_{1} + i\xi - s)} d\xi = \lim_{\sigma \to \infty} \int_{-\infty}^{\infty} \frac{\L\{\tau_h; \sigma + i\xi\}}{(\sigma + i\xi - s)} d\xi = 0 . \]
\end{proof}
	
 The following result shows that the mapping $\zeta \to e_{\zeta}$ is holomorphic on a suitable domain.
We write $\C_{-} = \{ s \in \C \mid \Real{s} < 0 \}$.

\begin{lemma}\label{LemmaAnalytic}
The vector-valued function 
	\begin{equation}
		\label{eq:HolomorphicMap} 
		\C_{-} \setminus R \to V_{\alpha, r}: \quad \zeta \to {e_{\zeta}} 
	\end{equation}
is holomorphic.
\end{lemma}

\begin{proof}
Set $U = \C_{-} \setminus R$. Clearly $e_{\zeta} \in V_{\alpha, r}$ for each $\zeta \in U$.
Consider $\tilde{e}_{\zeta}(x) = xe^{\zeta x}$.
We will show that the derivative of the mapping \eqref{eq:HolomorphicMap} is $\tilde{e}_{\zeta}$ for each $\zeta\in U$.
Fix some $\zeta \in U$ and let $\eta \in \C \setminus \{0\}$ be such that $\zeta + \eta \in U$.
We have to show that:
	\begin{equation}
		\label{eq:DerivFLinfty} 
		\left\| \frac{e_{\zeta + \eta}-e_{\zeta}}{\eta} - \tilde{e}_{\zeta} \right\|_{\alpha, r} \to 0, \quad \text{as } \eta \to 0 . 
	\end{equation}
First,
\[ \left\| \frac{e'_{\zeta + \eta}-e'_{\zeta}}{\eta} - \tilde{e}'_{\zeta} \right\|_{L^{\infty}}\leq \sup_{x \geq 0} |e^{(\zeta+\eta)x}-e^{\zeta x}| + |\zeta| \sup_{x \geq 0}\left| \frac{e^{(\zeta + \eta) x} - e^{\zeta x}}{\eta} - x e^{\zeta x} \right| . \]
Each of these terms converges to $0$ as $\eta\to0$. For the second one, since $\Real{\zeta} < 0$, there exists some $x_{\varepsilon} > 1$ such that $|xe^{\zeta x}| \leq \varepsilon / 4$ for all $x \geq x_{\varepsilon}$.
Then, for $\eta$ small enough with $\Real{\eta} \leq - \Real{\zeta}$,
\begin{align*}
    \sup_{x \geq 0}\left| \frac{e^{(\zeta + \eta) x} - e^{\zeta x}}{\eta} - x e^{\zeta x} \right|
    &= \sup_{x \geq 0} e^{x \Real{\zeta}} \left| \int_{0}^{x} (e^{\eta t} - 1) dt \right| \\
    &\leq x_{\varepsilon} \sup_{t \in [0, x_{\varepsilon}]} |e^{\eta t} - 1| + 2\sup_{x > x_{\varepsilon}} x e^{x \Real{\zeta}}
    \leq \varepsilon.
\end{align*}
Likewise, $\lim_{\eta\to0}  \sup_{x \geq 0} \left| e^{(\zeta + \eta) x}- e^{\zeta x} \right|=0$, establishing our first claim.
  
We now have $\L\{\tilde{e}_{\zeta}; s\} = 1/(s-\zeta)^{2}$. Hence,
\[
\sup_{s \in R} \left| \frac{\L\{e_{\zeta + \eta}; s\} - \L\{e_{\zeta}; s\}}{\eta} -  \L\{\tilde{e}_{\zeta}; s\} \right| 
=|\eta| \sup_{s \in R}\left| \frac{1}{(s - \zeta-\eta)(s - \zeta)^{2}} \right|
\to0, \qquad \mbox{as }\eta \to 0.
\]

This shows \eqref{eq:DerivFLinfty} and completes our proof.
\end{proof}

Next, we show that the  integral (\ref{IntegralTau}) is convergent in \(V_{\alpha,r}\).

\begin{lemma}\label{LemmaAbsolutelyConvergent}
    The vector-valued integral 
    \[\tau_h = \frac{1}{2 \pi i} \int_{\Gamma_h} \L\{\tau_h; \zeta\}  e_\zeta \:d\zeta\]
     is absolutely convergent in the Banach space $V_{\alpha, r}$ for $h > 0$.
\end{lemma}

\begin{proof}
Lemma \ref{LemmaAnalytic} implies that ${\Gamma_h} \ni \zeta \to \L\{\tau_h; \zeta\}e_{\zeta} \in V_{\alpha, r}$ is continuous. 
Therefore, it suffices to show that
	\[ \int_{\Gamma_h} |\L\{\tau_h; \zeta\}| \| e_{\zeta} \|_{\alpha, r} |d\zeta| < \infty . \]
In fact, we only need to verify that
	\[ \left( \int_{-\infty}^{-r'} + \int_{r'}^{\infty} \right) |\widehat{\tau}(\xi)| \| e_{-h + i\xi} \|_{\alpha, r} d\xi < \infty . \]
For $\zeta \in -h+i[ (-\infty,r')\cup (r', \infty)]$, we have
	\[ \| e_\zeta \|_{\alpha, r} \leq \frac{1}{r' - r} + h + |\Imaginary \zeta|, \]
whence the claim directly follows because $\tau \in \S(\R)$.
\end{proof}

We then find the next approximation result for $\tau_{h}$.

\begin{lemma}\label{LemmaApproximationbyX}
    Let $D \subseteq \{\zeta \in \C \mid \Real{\zeta} < \beta \}$ be such that it has at least one accumulation point on this half-plane. Then given any $h, \varepsilon > 0$, there are $\zeta_1, \ldots, \zeta_n \in D$ and $c_1,\ldots c_n \in \C$ such that 
    \[ \|\tau_h -\sum_{j=1}^n c_j e_{\zeta_j} \|_{\alpha,r} < \varepsilon.\]
\end{lemma}

\begin{proof}
Let $Z = \spn \{ e_{\zeta} \mid \zeta \in D \} \subseteq V_{\beta}$. We will show that $\tau_{h} \in \overline{Z}^{V_{\alpha, r}}$.
Let $\mu \in V_{\alpha, r}^{\prime}$ be a continuous linear functional and set $F(\zeta) = \langle \mu, e_{\zeta} \rangle $ for $\zeta \in \C_{-} \setminus R$. Then $F$ is holomorphic on this domain by Lemma \ref{LemmaAnalytic}. 
If $F(\zeta) = 0$ for every $\zeta \in D$, then it would be $0$ on its entire domain, and hence also on $\Gamma_h$. 
By Lemma \ref{LemmaAbsolutelyConvergent}, we thus conclude that 
	\[ \langle \mu, \tau_h \rangle = \langle \mu, \frac{1}{2 \pi i} \int_{\Gamma_{h}} \L\{\tau_h; \zeta\} e_{\zeta} \:d\zeta \rangle = \frac{1}{2 \pi i} \int_{\Gamma_{h}} \L\{\tau_h; \zeta\} F(\zeta) d\zeta = 0 . \]
As $\mu \in V_{\alpha, r}^{\prime}$ was chosen arbitrarily, it follows from the Hahn-Banach theorem that $\tau_{h} \in \overline{Z}^{V_{\alpha, r}}$.
\end{proof}

\begin{proof}[Proof of Proposition \ref{p:DensitySteps}]
	By Lemmas \ref{LemmaMisDense} and \ref{LemmaApproximationbyX}, we obtain in particular that  $V^{0}_{\alpha, r}$ is contained in the closure of $V_{\beta}$ in $V_{\alpha, r}$, which suffices to show the result.
\end{proof}

\begin{remark}
In fact, we have shown the stronger statement: $V^{0}_{\alpha, r}$ is contained in the closure of $\spn\{ e_{\zeta} \mid \zeta \in D \}$ in $V_{\alpha, r}$, for any $D$ as in Lemma \ref{LemmaApproximationbyX}.
\end{remark}

\section{The absence of remainders in the Ingham-Karamata theorem}
\label{sect absence of remainders}
We now switch our attention to the proof of Theorem \ref{th3 abserrorW-I}. The treatment here basically reproduces that from \cite[Section 3]{D-V2019} up to minor modifications\footnote{The proof in \cite{D-V2019} explicitly treats the optimality of the Wiener-Ikehara theorem (cf. Corollary \ref{OW-I}), which is basically equivalent to the optimality problem for the Ingham-Karamata theorem.}, but we give details for the sake of completeness and in order to emphasize the key role that Theorem \ref{MainResult} plays in the method.

We will make use of the following majorization lemma. We omit its proof and refer to \cite{D-V-NoteAbsenceRemWienerIkehara} for it.

\begin{lemma}[{\cite[Lemmas 1 and 2]{D-V-NoteAbsenceRemWienerIkehara}}]
\label{l:l1aofr} Let $\rho$ be a positive function such that $\rho(x)= o(1)$. Then, there exists $L \in C^1[0, \infty)$ such that 
\[
\lim_{x\to\infty} L(x)=\lim_{x\to\infty} L'(x)=0,
\]
\begin{equation}
\label{eq:1aofr}
\rho(x)=o(L(x)),
\end{equation}
and $\L\{L; s\}$ has analytic continuation to a sector $-\pi + \vartheta < \arg{s} < \pi - \vartheta$ for some angle $\vartheta \in (0, \pi / 2)$.
\end{lemma}

\begin{proof}[Proof of Theorem \ref{th3 abserrorW-I}] We reason by contradiction. Hence, as opposed to \eqref{eq:2I-K}, assume that there is some positive function $\rho$ such that $\rho(x)=o(1)$ and
\begin{equation}
\label{eq:2aofr}
\tau(x)=O(\rho(x)), \qquad  \mbox{for all }\tau\in V_{-\infty}.
\end{equation}
Majorizing $\rho$ if necessary, we might additionally assume that $\rho$ is continuous. We shall compare two locally convex topologies on $V_{-\infty}$. The first one is its canonical Fr\'{e}chet topology, already introduced in Section \ref{MainResults} by means of the family of norms \eqref{eq:ineq2}. By the assumption \eqref{eq:2aofr}, the family of norms 
\[ \| \tau \|'_{\alpha, r} = \| \tau \|_{\alpha, r} + \sup_{x\geq 0} \left| \frac{\tau(x)}{\rho(x)}\right| , \qquad \alpha<0 \mbox{ and } r > 0 , \]
is well-defined and clearly induces a second Fr\'{e}chet space topology on $V_{-\infty}$. The second of these topologies is obviously finer than the first one, so that the open mapping theorem for Fr\'{e}chet spaces allows us to conclude that they must actually coincide. Consequently, there are $\alpha<0$, $r > 0$, and a constant $C > 0$ such that
\begin{equation}
\label{eq:3aofr}
 \sup_{x \geq 0} \left|\frac{\tau(x)}{ \rho(x)}\right| \leq C \| \tau \|_{\alpha, r} \:,
\end{equation}
for all $\tau\in V_{-\infty} .$ Armed with Theorem \ref{MainResult}, we conclude that this inequality extends to the whole of $V_{\alpha, r}$.

We now fix $L$ having the properties stated in Lemma \ref{l:l1aofr}. The next step is using \eqref{eq:3aofr} to produce a conflict with \eqref{eq:1aofr}. For it, we consider the function
	\[ L_{b}(x) = L(x) \sin(bx), \qquad b>0, \]
whose Laplace transform is given by
	\[ \mathcal{L}\{L_b; s\} = \frac{1}{2i} \left( \L\{L; s - ib \} - \L\{L; s + ib\} \right) . \]
From this expression one sees that there exists $M > 0$ such that $L_{b} \in V_{\alpha, r}$ for $b \geq M$ and furthermore $\{ \|L_{b}\|_{\alpha, r} \mid b \in [M, M + 1] \}$ is a bounded subset of the Banach space $V_{\alpha, r}.$
For $x$ large enough, $\sin(b x) = 1$ for some $b \in [M, M + 1]$. Hence,
	\[ \frac{L(x)} {\rho(x)} \leq \sup_{b \in [M, M + 1]} \left| \frac{ L(x) \sin(bx)}{\rho(x)} \right| \leq \sup_{b \in [M, M + 1]} \sup_{y \geq 0} \left| \frac{ L_{b}(y)}{\rho(y)}\right| \leq C \sup_{b \in [M, M + 1]} \|L_{b}\|_{\alpha, r} < \infty , \]
for all sufficiently large $x$. Thus, $L(x) = O(\rho(x))$, in contradiction with \eqref{eq:1aofr}.
\end{proof}

\end{document}